\documentclass{amsart}

\theoremstyle{plain}
\newtheorem{Prop}{Proposition}[section]
\newtheorem{Thm}[Prop]{Theorem}
\newtheorem{Cor}[Prop]{Corollary}

\newtheorem{Lem}[Prop]{Lemma}

\theoremstyle{definition}
\newtheorem{Def}[Prop]{Definition}

\theoremstyle{remark}
\newtheorem{Rem}[Prop]{Remark}

\newtheorem{Problem}[Prop]{\bf Problem}

\def\dim{\mathop{\roman{dim}}}
\def\int{\mathop{\roman{int}}}

\def\1{^{-1}}

\def\dim{\text{dim}}
\def\asdim{\text{asdim}}

\def\T2{{\mathbf T_2}}
\def\NN{{\mathbb N}}

\def\UU{{\mathcal U}}

\def\RR{{\mathbb R}}
\def\ZZ{{\mathbb Z}}

\errorcontextlines=0 \numberwithin{equation}{section}

\begin{document}

\title[Assuad-Nagata dimension of nilpotent groups with arbitrary left invariant metrics]{Assouad-Nagata dimension of nilpotent groups with arbitrary left invariant metrics}
\author{J.~Higes}
\address{Departamento de Geometr\'{\i}a y Topolog\'{\i}a,
Facultad de CC.Matem\'aticas. Universidad Complutense de Madrid.
Madrid, 28040 Spain}
\email{josemhiges@yahoo.es}

\keywords{Assouad-Nagata dimension, asymptotic dimension, nilpotent groups}

\subjclass[2000]{Primary 54F45; Secondary 55M10, 54C65}
\date{April 23, 2008}
\thanks{ The author is supported by Grant AP2004-2494 from the Ministerio de Educaci\' on y Ciencia, Spain. He also thanks Jerzy Dydak and N. Brodskyi for helpful comments 
and support.}

\begin{abstract}Suppose $G$ is a countable, not necessarily finitely generated, group. We show $G$ admits a proper, left-invariant metric $d_G$ 
such that the Assouad-Nagata dimension of $(G,d_G)$ is infinite, provided the center of $G$ is not locally finite. 
As a corollary we solve two problems of A.Dranishnikov.
\end{abstract}
\maketitle
\tableofcontents
\section{Introduction}\label{Introduction}
The asymptotic dimension was introduced by Gromov in \cite{Gromov} as a coarse invariant to study the geometric structure of finitely generated groups. 
 We refer to \cite{Bell-Dran} for a survey about this topic. Closely related with the asymptotic dimension is the asymptotic dimension of linear type.
It is also called asymptotic Assouad-Nagata dimension in honor of Patrice Assouad who introduced it in \cite{Assouad} from the ideas of Nagata. 
Such dimension can be considered as the linear version of the asymptotic dimension. In recent years a part of the research 
activity was focused on this dimension and its relationship with the asymptotic dimension (see for example \cite{Lang}, \cite{Dran-Smith2}, \cite{Dran-Zar}, \cite{Brod-Dydak-Higes-Mitra},
\cite{Brod-Dydak-Higes-Mitra Lipschitz}, \cite{Brod-Dydak-Lang}, \cite{Brod-Dydak-Levin-Mitra}, \cite{Nowak} \cite{Dydak-Higes}, \cite{Higes}). One of the main problems of 
interest consists in studying the differences between the asymptotic dimension and the asymptotic Assouad-Nagata dimension in the context of the geometric group theory. 
In particular
there are two main questions:
\begin{enumerate}
\item Given a finitely generated group $G$ with a word metric $d_G$. Are the asymptotic dimension and the asymptotic Assouad-Nagata dimension of $(G, d_G)$ equal?
\item In the case the two dimensions differ, must their difference be infinite?
\end{enumerate}  
It is known that the first question has an affirmative answer for abelian groups,
finitely presented groups of asymptotic dimension one, and for hyperbolic groups. But in general the 
answer is negative. Nowak (\cite{Nowak}) found for every $n\ge 1$ a finitely generated group of asymptotic dimension $n$ but of infinite asymptotic Assouad-Nagata dimension. 
\par
As far as the author knows the second question is still open. 
There is no example of a finitely generated group such that the asympotic dimension is strictly smaller than  the asymptotic Assouad-Nagata 
dimension but both are finite. In \cite{Higes} the second question was solved in a more general context. It was proved that for every $n$ and $m$ there exists a countable 
abelian group(non finitely generated) with a proper left invariant metric such that the group is of asymptotic dimension $n$ but of asympotic Assouad-Nagata dimension equal to 
$n+m$. Proper left invariant metrics 
are natural generalizations of word metrics. Therefore it is natural to ask the same
two problems in the case of finitely generated groups equipped with proper left invariant
metrics.\par
The aim of this note is to study the behaviour of the asymptotic Assouad-Nagata dimension in nilpotent groups with proper left invariant metrics. 
It is highly likely that both dimensions coincide 
in nilpotent groups with word metrics (see \cite{DranProblems}). For example in \cite{Dydak-Higes} it was proved their coincidence for the Heisenberg group. We will show that 
for every nilpotent group it is possible to find a proper left invariant metric such that both dimensions are different.  Dranishnikov in 
\cite{DranProblems} asked about what can be considered a special case:
\begin{Problem} (Dranishnikov \cite{DranProblems}) Does $dim_{AN}(\ZZ, d) = 1$ for every left invariant metric on $\ZZ$?
\end{Problem}  
In relation to the above problem he asked the following:
\begin{Problem}(Dranishnikov \cite{DranProblems}) Does $dim_{AN}(\Gamma \times \ZZ) = dim_{AN}(\Gamma) +1$ for any left invariant metric on $\ZZ$?
\end{Problem}
Notice that in \cite{Dran-Smith2}, Dranishnikov and Smith proved that $asdim_{AN}(G \times \ZZ) = \asdim_{AN} (G) +1$ for every finitely generated group $G$ but in this case 
the metrics consideres were the word metrics.\par
Our main theorem deals with a larger class of groups than nilpotent ones:
\begin{Thm} If $G$ is a group such that its center is not locally finite then there exists a proper left invariant metric $d_G$ such that $asdim_{AN}(G, d_G) = \infty$.
\end{Thm}
It is clear that if we apply the previous theorem to $G = \ZZ$, the two questions of
Dranishnikov are solved in negative. \par
The key ingredient of the proof was introduced by the author in \cite{Higes}.  
In such paper it was shown that if there exists a sequence of 
isometric embeddings (up to dilatation) of balls $\{B(0, k_i)\}_{i\in \NN}$ of $\ZZ^m$ 
into a metric space $X$ where the sequence of radious $k_i$ tends to infinity 
then the asymptotic Assouad-Nagata dimension of $X$ is greater than $m$. That result should be viewed as applying the philosophy of Whyte \cite{Whyte} 
in a rather restricted form. 
Instead of looking for subsets of a group $G$ that are bi-Lipschitz equivalent to $Z^n$, 
we are constructing groups that contain rescaled copies of large balls in $Z^n$. \par
Section \ref{GeometricAssouadNagata}  is devoted to a key ingredient used in Section \ref{MainSection} to present proofs of main results.

\section{Preliminaries}\label{Preliminaries}
Let $s$ be a positive real number. An {\it $s$-scale chain} (or $s$-path)
 between two points $x$ and $y$ of a metric space $(X, d_X)$ is defined as
a finite sequence points
$\{x= x_0, x_1, ..., x_m = y\}$ such that
$d_X(x_i, x_{i+1}) < s$ for every $i = 0, ..., m-1$.
A subset $S$ of a metric space $(X, d_X)$ is said to be {\it $s$-scale connected} if there exists an $s$-scale chain
contained in $S$ for every two elements of $S$.
\begin{Def} A metric space $(X, d_X)$ is said to be of
{\it asymptotic dimension} at most $n$ (notation $\asdim(X, d) \le
n$) if there is an increasing function $D_X: \RR_+ \to \RR_+$ such
that for all $s> 0$ there is a cover $\UU =\{\UU_0, ...,\UU_n\}$
so that the $s$-scale connected components of each $\UU_i$ are
$D_X(s)$-bounded i.e. the diameter of such components is bounded
by $D_X(s)$. \par The function $D_X$ is called an {\it
$n$-dimensional control function} for $X$. Depending on the type
of $D_X$ one can define the following two invariants: \par A
metric space $(X, d_X)$ is said to be of {\it  Assouad-Nagata
dimension} at most $n$ (notation $\dim_{AN} (X, d) \le n$) if it
has an $n$-dimensional control function $D_X$ of the form
 $D_X(s) = C\cdot s$ with $C>0$ some fixed constant. \par
A metric space $(X, d_X)$ is said to be of
{\it asymptotic Assouad-Nagata dimension} at most $n$ (notation $\asdim_{AN} (X, d) \le n$) if it has an
$n$-dimensional control function $D_X$ of the form
 $D_X(s) = C\cdot s +k$ with $C >0$ and $k\in \RR$ two fixed constants.
\end{Def}
It is clear from the definition that for every metric space $(X, d_X)$, $asdim(X, d_X) \le asdim_{AN}(X, d_X)$.\par
One important fact about the asymptotic dimension is that it is invariant under coarse equivalences. 
Given a map $f: (X, d_X) \to (Y, d_Y)$ between two metrics spaces
it is said to be a {\it coarse embedding} if there exist two increasing functions 
$\rho_+: \RR_+ \to \RR_+$ and $\rho_-: \RR_+\to \RR_+$ with 
$\lim\limits_{x\to \infty} \rho_-(x) =\infty$
such that:
\[\rho_-(d_X(x, y)) \le d_Y(f(x), f(y))\le \rho_+(d_X(x, y)) \text{ for every } x, y \in X.\]
Now a {\it coarse equivalence}
 between two metrics spaces $(X, d_X)$ and $(Y, d_Y)$ is defined as a coarse embedding
$f: (X, d_X)\to (Y, d_Y)$ for which there exists a constant 
$K >0$ such that
$d_Y(y, f(X))\le K$ for every $y \in Y$. 
If there exists a coarse equivalence between $X$ and $Y$ both spaces are said to be {\it coarsely equivalent}.\par
The metrics spaces in which we are interested are countable groups with proper left invariant metrics.
\begin{Def}
A metric $d_G$ defined in a group $G$ is said to be a
{\it proper left invariant metric} if it satisfies the following conditions:
\begin{enumerate}
\item $d_G(g_1\cdot g_2, g_1\cdot g_3) = d_G(g_2, g_3)$ for every $g_1, g_2, g_3 \in G$.
\item For every $K>0$ the number of elements $g$ of $G$ such that $d(1_G, g)_G \le K$ is finite.
\end{enumerate}
\end{Def}

The following basic result of Smith will be used to get the main theorem. 
\begin{Thm}(Smith \cite{Smith})\label{SmithTheorem} Two proper left invariant metrics defined in a countable group are coarsely equivalent.
\end{Thm}

One way of constructing proper left invariant metrics in a countable group is via proper norms. 
\begin{Def}
A map $\|\cdot\|_G: G \to \RR_+$ is called to be a {\it proper norm} if  it satisfies
the following conditions:
\begin{enumerate}
\item $\|g\|_{G} = 0$ if and only if $g = 1_G$.
\item $\|g\|_G = \|g^{-1}\|_G$ for every $g \in G$.
\item $\|g\cdot h\|_G \le \|g\|_G +\|h\|_G$ for every $g, h \in G$.
\item For every $K>0$ the number of elements of $G$ such that $\|g\|_G\le K$ is finite.
\end{enumerate}
\end{Def}
It is clear that there is a one-to-one correspondence between proper norms and proper left invariant metrics. \par
Now we are interested in methods to get proper norms with some special properties. In general this task is not easy. In this paper 
we will use the method of weights described by Smith in \cite{Smith}.  Let $S$ be a symmetric system of
generators(possibly infinite) of a countable group $G$ and let $\omega:
L \to \RR_+$ be a function({\it weight function} or {\it system of weights}) that satisfies:
\begin{enumerate}
\item $\omega(s) = 0$ if and only if $s = 1_G$
\item $\omega(s) = w(s^{-1})$.
\item $\omega^{-1}[0, N]$ is finite for every $N$ .
\end{enumerate}
Then the function $\|\cdot\|_w: G\to \RR_+$ defined by:
\[\|g\|_w = \min\{\sum_{i=1}^n w(s_i)| x = \Pi_{i=1}^n s_i, \text{ }s_i \in S\}\]
is a proper norm. Such norm will be called the proper norm {\it  generated by the system of weights} $\omega$ and the associated left invariant metric
will be the {\it left invariant metric generated by the system of weights} $\omega$. 
\begin{Rem}\label{Weight}
\begin{enumerate} 
\item If we define $w(g) = 1$ for
all the elements $g \in S$ of a finite generating system $S\subset G$ ($G$ a finitely generated group) we will obtain the usual word metric.
\item Notice that if we have a proper norm $\|\cdot \|_G$ in a countable group $G$ and we take the system of weiths defined by
$\omega(g) = \|g\|_G$ then the proper norm $\|\cdot\|_{\omega}$ generated by this system of weights coincides with $\|\cdot\|_G$.
\item We can construct easily integer valued proper left invariant metrics by getting weight functions with integer range. 
\end{enumerate}
\end{Rem}

\section{Lower bounds for Assouad-Nagata dimension}\label{GeometricAssouadNagata}
The aim of this section is to give a sufficient condition in a metric space $(X, d_X)$ that implies $asdim_{AN}(X, d_X) \ge n$ for some $n$.
To get that condition we will use the notion of asymptotic cone.\par
Let $(X, d_X)$ be a
metric space. Given a non-principal ultrafilter $\omega$ of $\NN$
and a sequence $\{x_n\}_{n\in\NN}$ of points of $X$, the {\it
$\omega$-limit} of $\{x_n\}_{n\in\NN}$ (notation:
$\lim\limits_{\omega} x_n = y$) is an element $y$ of $X$
such that for every neighborhood $U_y$ of $y$ the set $F_{U_y} = \{n|
x_n \in U_y\}$ belongs to $\omega$. It can be proved
easily that the $\omega$-limit of a sequence always exists in a compact space.
\par Assume $\omega$ is a non principal ultrafilter of $\NN$.  Let
$d =\{d_n\}_{n\in\NN}$ be an
$\omega$-divergent sequence of positive real numbers  and 
let $c = \{c_n\}_{n \in \NN}$ be  any
sequence of elements of $X$. Now we can construct the
{\it asymptotic cone} (notation: $Cone_{\omega}(X, c, d)$) of $X$ as follows:
\par Firstly define the set of all sequences $\{x_n\}_{n\in\NN}$
of elements of $X$ such that $\lim\limits_{\omega}\frac{d_X(x_n,
c_n)}{d_n}$ is bounded. In such set take the pseudo metric given
by: \[D(\{x_n\}_{n\in\NN}, \{y_n\}_{n\in\NN}) =
\lim\limits_{\omega}\frac{d_X(x_n, y_n)}{d_n}.\] By identifying
sequences whose distances is $0$ we get the metric space
$Cone_{\omega}(X, c, d)$.\par 
Asymptotic cones were firstly
introduced by Gromov in \cite{Gromov}. There has been a lot of
research relating properties of groups with topological properties
of its asymptotic cones. For example a finitely generated group is
virtually nilpotent if and only if all its asymptotic cones are
locally finite \cite{Gromov2} or a group is hyperbolic if and only if all of its asymptotic cones are $\RR$-trees (\cite{Gromov} and \cite{Drutu}). \par
In \cite{Dydak-Higes} it was shown 
the following relationship between 
the topological dimension of an asymptotic cone and the asymptotic Assouad-Nagata dimension of the space:

\begin{Thm}\label{DimensionOfCones} [Dydak, Higes \cite{Dydak-Higes}]
$dim(Cone_{\omega}(X, c, d) \le dim_{AN}(Cone_{\omega}(X, c, d) \le asdim_{AN}(X, d_X)$
for any metric space $(X, d_X)$ and every asymptotic cone $Cone_{\omega}(X, c, d)$.
\end{Thm}

We recall now the following:
 
\begin{Def}
A function $f: (X, d_X) \to (Y, d_Y)$ between metric spaces is said to be a dilatation if there exists a constant $C \ge 1$ such that 
$d_Y(f(x), f(y)) = C \cdot d_X(x, y)$ for every $x, y \in X$.\par
The number $C$ will be called the {\it dilatation constant}.
\end{Def}

In the following proposition dilatations will be from balls of $\ZZ^n$ with the $l_1$-metric to general metric spaces. 

\begin{Prop}\label{SequenceOfCubesImpliesCube} Let $(X, d_X)$ be a
metric space and let $\{k_m\}_{m \in \NN}$ be an increasing sequence of natural numbers. If for some $n \in \NN$ there is a sequence of 
dilatations $\{f_m\}_{m =1}^{\infty}$ of the form $f_m: B^n(0, k_m) \to (X,d_X)$ with $B^n(0, k_m)\subset \ZZ^n$ the ball of radious $k_m$ then 
there exists an asymptotic cone $Cone_{\omega}(X, c, d)$ of $(X, d_X)$ such that $[-1, 1]^n \subset Cone_{\omega}(X, c, d)$.
\end{Prop}
\begin{proof}
Suppose given $(X, d_X)$ and $\{k_m\}_{m\in \NN}$ as in the hypothesis. 
Let us prove firstly the case $n = 1$.  Assume that $\{C_m\}_{m\in \NN}$ is the sequence of dilatation constants of $\{f_m\}_{m\in\NN}$.  Take $\omega$ some ultrafilter
of $\NN$ and define $c = \{f_m(0)\}_{m=1}^{\infty}$ and $d = \{d_m\}_{m= 1}^{\infty}$ with  $d_m = C_m \cdot k_m$. We will prove that $Cone_{\omega}(X, c, d)$ 
contains $[-1,1]$. For each $t \in [-1, 1]$ let $A_m^t$ be the subset  
$\{-k_m,..., k_m\}$  such that $x \in A_m^t$ if and only if the distance between $\frac{C_m\cdot x}{d_m}$ and $t$ is 
minimum. Notice that this implies that the distance between $C_m \cdot x$ and $d_m \cdot t$ is less thatn $C_m$. Take now
the sequence $\{r_m^t\}_{m = 1}^{\infty}$  where $r_m^t$ is the
infimum of $A_m^t$.\par
Define the map $g:[-1, 1] \to Cone_{\omega}(X, c, d)$ by $g(t) = x^t$
if the sequence $\{f_m(r_m^t)\}_{m = 1}^{\infty}$ is in the class $x^t$. As:
\[\lim\limits_{\omega} \frac{d(f_m(0), f_m(r_m^t))}{d_m} =
\lim\limits_{\omega} \frac{C_m\cdot |r_m^t|}{d_m} \le \lim\limits_{\omega} \frac{C_m \cdot k_m}{d_m} = 1\]
the map is well defined. Let us prove it is an
isometry. From the definition of $r_m^t$ we get that if $t_1 < t_2$ then
$r_m^{t_1}\le r_m^{t_2}$ what implies $\lim\limits_{\omega}\frac{d(f_m(r_m^{t_1}), f_m(r_m^{t_2}))}{d_m}
= \lim\limits_{\omega} \frac{C_m(r_m^{t_2}-r_m^{t_1})}{d_m}$.
So the unique thing we need to show is that $\lim\limits_{\omega} \frac{C_m \cdot r_m^t}{d_m} = t$
for every $t$. Notice that we have $\lim\limits_{\omega}\frac{C_m}{d_m} = 0$ as 
$\lim\limits_{\omega}k_m = \infty$ but $\lim\limits_{\omega} \frac{C_m \cdot k_m}{d_m} = 1$. This implies that  
given $\epsilon >0$ there exists $G_{\epsilon} \in \omega$ such that $\frac{C_m}{d_m}<\epsilon$ for every $m \in G_{\epsilon}$. Therefore 
by the choice of $r_m^t$  
if $m \in G_{\epsilon}$ we have $|C_m\cdot r_m^t -d_m\cdot t|<  C_m$ and then $|\frac{C_m \cdot r_m^t}{d_m} - t|< \frac{C_m}{d_m}\le \epsilon$.\par
Now let us do the general case. Let $(s_1, ..., s_n) \in [-1, 1]^n$. By the previous case we get that for every $j = 1,..., n$
 there exists a sequence $\{r_m^{s_j}\}_{m \in \NN}$ with $r_m^{s_j}\in \{-k_m,...,k_m\}$ such that 
$\lim\limits_{\omega} \frac{C_m\cdot r_m^t}{d_m} = s_j$. In a similar way as before we construct a map $g: [-1,1]^n \to Cone_{\omega}(X, c, d)$ 
by defining $g(s_1, ..., s_m)$ as the class that contains the sequence $\{f_m(r_m^{s_1}, ..., r_m^{s_n})\}_{m = 1}^{\infty}$. To finish the proof it will be enough to check that
 for every $s, t \in [-1, 1]^n$ with $s = (s_1,..., s_n)$ and $t = (t_1,..., t_n)$, the following equality holds:
\[   \lim\limits_{\omega}\frac{d_X(f_m(r_m^{s_1},..., r_m^{s_n}), f_m(r_m^{t_1},..., r_m^{t_n}))}{d_m} = \sum_{i= 1}^n |s_i -t_i|\]
As $f_m$ is a dilatation of constant $C_m$ we can write:
\[\lim\limits_{\omega}\frac{d_X(f_m(r_m^{s_1},..., r_m^{s_n}), f_m(r_m^{t_1},..., r_m^{t_n}))}{d_m} = \sum_{i= 1}^n \lim\limits_{\omega} 
\frac{C_m\cdot |r_m^{s_i}- r_m^{t_i}|}{d_m}\]
And again by the case $n =1$ we can deduce that the last term satisfies the equality: 
 \[\sum_{i= 1}^n \lim\limits_{\omega} 
\frac{C_m\cdot |r_m^{s_i}- r_m^{t_i}|}{d_m} = \sum_{i=1}^n |s_i-t_i|\].
\end{proof}

From the previous proposition we can get the following result that is one of the ingredients of the main theorem.
 
\begin{Cor}\label{SequenceImpliesNagata}
If for metric space $(X, d_X)$ and form some $n \in \NN$ there exists a sequence of dilatations $f_m : B^n(0 , k_m) \to (X, d_X)$ with
$\lim\limits_{m\to \infty} k_m = \infty$ and $B^n(0, k_m)\subset \ZZ^n$ the ball of radious $k_m$ then $asdim_{AN}(X, d_X) \ge n$.
\end{Cor}
\begin{proof}
By  proposition \ref{SequenceOfCubesImpliesCube} we get that there exists an asymptotic cone of $X$ such that $[-1, 1]^n \subset Cone_{\omega}(X, c, d)$.
Applying theorem \ref{DimensionOfCones} we obtain immediately:
\[n \le dim(Cone_{\omega}(X, c, d))\le dim_{AN}(Cone_{\omega}(X, c, d)) \le asdim_{AN}(X, d_X)\]
\end{proof}

\section{Main results}\label{MainSection}
The idea of the proof of the main theorem consists in creating a metric inductively. In each step we will construct a new metric that satisfies two conditions. 
First condtion says that the new metric doer not change a sufficiently large ball of the old metric. Second condtion implies there is a dilatation from some 
sufficiently large ball of $\ZZ^n$ into the group with the new metric. In fact that dilatation will be the restriction of some homomorphism $f: \ZZ^n \to G$. 
Then we will apply corollary \ref{SequenceImpliesNagata}. The following lemma could be considered as the 
induction step.   
\begin{Lem}\label{Inductive}Let $G$ be a finitely generated group such that its center is not locally finite. Let $d_G$ be a proper left invariant metric.  
In such condtions for every $k, m, R \in \NN$ 
there exists a proper left invariant metric $d_{\omega}$ that satisfies the following conditions:
\begin{enumerate}
\item $\|g\|_{\omega} \le \|g\|_G$.
\item $\|g\|_{G} = \|g\|_{\omega}$ if $\|g\|_{\omega}\le R$. 
\item There is an homomorphism $f: \ZZ^m \to G$ such that the restriction $f|_{B(0,k)}$ of $f$ to the ball radious $k$ is a dilatation in $(G, d_{\omega})$.
\end{enumerate} 
\end{Lem}
\begin{proof} 
Suppose $k, m$ and $R$ given and let $a$ and $C$ be two natural numbers that satisfy:
\[R < C < \frac{a}{2\cdot k\cdot m^2}.\]
As the center of $G$ is not locally finite there exists an element $g$ in the center of infinite order. The restriction of the metric $d_G$ to the subgroup generated by $g$ 
defines a proper left invariant metric in $\ZZ = < g >$. By theorem \ref{SmithTheorem} we know that two proper left invariant 
metrics defined in a group are coarsely equivalent, hence 
we can find an integer number $h_1 \in \ZZ$ such that if $|h| \ge |h_1|$ then $\|g^{h}\|_G \ge a$. Let $p_1 =1$  and for every $j = 2...m$ we define $p_j$ 
as a sufficiently large 
number that satisfies $\sum_{i=1}^{j-1} (2\cdot k\cdot m)2^{p_i} < 2^{p_j}$. Take now the finite set of integer numbers $\{h_1,..., h_m\}$ with $h_j = 2^{p_j}\cdot h_1$ for 
every $j = 2,..., m$. In this situation we create the norm $\|\cdot\|_{\omega}$ generated by the following system of weights:
\[\omega(z) = \begin{cases} \|z\|_G \text{ if } z \ne g^{\pm h_i} \text{ for every } i = 1...m \\ C \text{ otherwise } \end{cases}\]
By the choice of $C$ and $\{h_1,..., h_m\}$ it is clear that the two first conditions of the lemma are satisfied. To prove the 
third condition we define the homomorphism $f: \ZZ^m \to G$ as $f(x_1,..., x_m) = g^h$ with $h = \sum_{i=1}^m x_i \cdot h_i$. 
Let us show that the restriction 
$f|_{B(0, k)}: B(0, k) \to G$ to the ball of radious $k$ is a dilatation of constant $C$. It will be enough to check that:
\[\|g^h\|_{\omega} = \sum_{i = 1}^m |x_i| \cdot C \text{ if } h = \sum_{i=1}^m x_i \cdot h_i \text{ and } |x_i| \le k\]
The reasoning will be by contradiction. Suppose there exists an element of the form $g^h$ with $h = \sum_{i= 1}^{m} x_i\cdot h_i$ and $|x_i| \le k$ 
such that $\|g^h\|_{\omega} < \sum_{i = 1}^m |x_i| \cdot C$. 
This implies that there exist and $r = \sum_{i= 1}^m y_i \cdot h_i$ and an $s \in G$ such that 
$g^h = g^r\cdot s$ and:
\[\|g^h\|_{\omega} = \sum_{i=1}^m |y_i| \cdot C + \|s\|_G. \]
Notice that $|y_i|\le k\cdot m$. There are now two possible cases:\par
Case $s = 1_G$: In this situation we have $\sum_{i=1}^m |y_i| \cdot C < \sum_{i=1}^m |x_i| \cdot C$ so there exists an $i$ such that $x_i \ne y_i$. Let $j = 
\max\{i | x_i \ne y_i\}$. From the fact $g^h = g^r$ we can deduce $(x_j - y_j)\cdot h_j = \sum_{i=1}^{j-1} (y_i- x_i) \cdot h_i$ it means $(x_j-y_j) \cdot 2^{p_j}\cdot h_1 = 
\sum_{i=1}^{j-1}(y_i- x_i) \cdot 2^{p_i}\cdot h_1$ and then:
\[2^{p_j} \le |x_j-y_j| \cdot 2^{p_j}\le \sum_{i=1}^{j-1}|y_i- x_i| \cdot 2^{p_i} \le \sum_{i=1}^{j-1}|y_i|+ |x_i| \cdot 2^{p_i} \le \sum_{i=1}^{j-1}2^{p_i}\cdot (k \cdot m+ k)
< 2^{p_j}\]
A contradiction. Therefore the first case is not possible.\par
Case 2: $s \ne 1_G$. In this case we have $g^{(h-r)}  = s$ and $h-r \ne 0$ what implies $|h-r| \ge |h_1|$ and hence $\|s\|_G \ge a$. But from the fact 
$\|g^h\|_{\omega} = \sum_{i=1}^m |y_i| \cdot C + \|s\|_G < \sum_{i = 1}^m |x_i| \cdot C$ we can deduce:
\[a \le \|s\|_G < \sum_{i=1}^m (|x_i|-|y_i|) \cdot C \le \sum_{i=1}^m (k + k\cdot m) \cdot C \le 2\cdot k \cdot m^2 \cdot C.\]
Therefore $C \ge \frac{a}{2\cdot k\cdot m^2}$ and this contradicts the choice of $a$ and $C$.
\end{proof}
Here is the main theorem.
\begin{Thm}\label{MainTheorem} If $G$ is a group such that its center is not locally finite then there exists a proper left invariant metric $d_G$ such that $asdim_{AN}(G, d_G) = \infty$.
\end{Thm}
\begin{proof}
We will use corollary \ref{SequenceImpliesNagata} and the previous lemma. Take any integer valued proper left invariant metric $d$ in $G$ (see remark \ref{Weight}) 
and some increasing sequences 
$\{k_i\}_{i\in \NN}$ and $\{M_i\}_{i\in \NN}$ 
of natural numbers. Let us construct the metric $d_G$ of the theorem by an inductive process. \par
Step 1: Apply the previous lemma to $d$ with $k = k_1$, $m = 1$ and $R = M_1$. We obtain a proper left invariant metric $d_{\omega_1}$ such that the ball $B_{\omega_1}(1_G, R_1)$
is equal to the ball of radious $R_1$ of $d$. Also there exists a dilatation $f: B(0, k_1) \to G$ from the ball or radious $k_1$ of $\ZZ$ to $G$.\par
Induction Step: Suppose now that we have contructed a finite sequence of proper left invariant metrics $ L = \{d_{\omega_1}, ..., d_{\omega_n}\}$ and a finite sequence of natural numbers 
$R_1 < R_2 <...< R_n$ that satisfy the following conditions:
\begin{enumerate}
\item $\|g\|_{\omega_i} \le \|g\|_{\omega_{i-1}}$.
\item $\|g\|_{\omega_{i}} = \|g\|_{\omega_{i-1}}$ if $\|g\|_{\omega_{i}} \le R_i$
\item There exists an homomorphism $f_i: \ZZ^i \to G$ such that the restriction $f_i|_{B(0,k_i)}$ is a dilatation in $(G, d_{\omega_i})$ for every $i = 1,...,n$
\item $diam(f_i(B(0, k_{i})) < R_{i+1}$ for every $i = 1,...,n-1$ 
\end{enumerate} 
In these conditions define $R_{n+1} = \max\{M_{n+1}, R_n+1, diam(f_n(B(0, k_n))\}$ and apply the previous lemma to $d_{\omega_n}$ with $k = k_{n+1}$, $m = n+1$ and
$R = R_{n+1}$. We have now a new proper left invariant metric $d_{\omega_{n+1}}$. It is clear that the new finite sequence of proper left invariant metrics $\{d_{\omega_{n+1}}\}\cup 
L$ and the new finite sequence of numbers $R_1 < R_2 <...< R_n < R_{n+1}$ satisfy the same four conditions. 
\par 
Repeating this procedure we construct a sequence of integer 
valued proper left invariant 
metrics $\{d_{\omega_i}\}_{i=1}^{\infty}$ and an increasing sequence of natural numbers $\{R_i\}_{i=1}^{\infty}$.  By the first two properties and the fact 
$\lim_{i\to \infty} R_i = \infty$
we deduce that 
for every $g\in G$ the sequence $\{\|g\|_{\omega_i}\}_{i=1}^{\infty}$ is asymptotically constant. Define now
 the function 
$\|\cdot\|_G: G \to \NN$ by $\|g\|_G =\lim_{i \to \infty} \|g\|_{\omega_i}$. Again by the first two properties we can check $\|\cdot\|_G$ is a proper norm. So it 
defines a proper left invariant metric $d_G$.  Using 
the third and fourth properties 
we have that for every $i\in \NN$ there exists an homomorphism $f_i: \ZZ^i \to G$ such that the restriction to the ball $B(0, k_i)$ is a dilatation in $(G, d_G)$. 
As in each step we are increasing the dimension of the balls, we get that for every $m \in \NN$ the metric space $(G, d_G)$ satisfies the conditions of corollary \ref{SequenceImpliesNagata} so we get $asdim_{AN}(G, d_G) \ge m$.
Therefore $asdim_{AN}(G, d_G) = \infty$.  
\end{proof}

Recall that in \cite{Dran-Smith} Dranishnikov and Smith showed that the asymptotic dimension of finitely generated 
nilpotent groups is equal to the hirsch length of the group. Hence the asymptotic dimension of a nilpotent group is always finite for 
every proper left invariant metric. Next trivial corollary shows that the unique nilpotent groups that satisfy the same property for the asymptotic 
Assouad-Nagata dimension are finite.  

\begin{Cor} Let $G$ be a finitely generated nilpotent group. $G$ is non finite if and only if 
there exists a proper left invariant metric $d_{\omega}$ defined in $G$ such that $asdim_{AN}(G, d_{\omega}) = \infty$.
\end{Cor}
\begin{proof}
For one implication just use that the asymptotic Assouad-Nagata dimension is zero for all bounded spaces. The other implication is a particular case
 of the main theorem.
\end{proof}

\end{document}